\documentclass[11pt]{amsart}
\usepackage[utf8]{inputenc}

\usepackage[english]{babel}
\usepackage{amsmath, amssymb, amsthm, amscd,color,comment}
\usepackage{cancel}
\usepackage{cite}
\usepackage{alltt}
\usepackage[dvipsnames]{xcolor}
\usepackage{array}
\usepackage[small,bf,labelsep=period]{caption}
\usepackage{mathtools}
\usepackage{hyperref}

\makeatletter
\def\moverlay{\mathpalette\mov@rlay}
\def\mov@rlay#1#2{\leavevmode\vtop{%
		\baselineskip\z@skip \lineskiplimit-\maxdimen
		\ialign{\hfil$\m@th#1##$\hfil\cr#2\crcr}}}
\newcommand{\charfusion}[3][\mathord]{
	#1{\ifx#1\mathop\vphantom{#2}\fi
		\mathpalette\mov@rlay{#2\cr#3}
	}
	\ifx#1\mathop\expandafter\displaylimits\fi}
\makeatother

\usepackage{pgfplots}
\pgfplotsset{compat=1.15}
\usepackage{mathrsfs}
\usetikzlibrary{arrows}

\usepackage{enumerate}
\usepackage{cancel}
\newtheorem{theorem}{Theorem}[section]

\newtheorem{definition}[theorem]{Definition}
\newtheorem{example}[theorem]{Example}

\newtheorem{remark}[theorem]{Remark}
\newtheorem{lemma}[theorem]{Lemma}
\newtheorem{corollary}[theorem]{Corollary}
\newtheorem{proposition}[theorem]{Proposition}

\DeclarePairedDelimiter\floor{\lfloor}{\rfloor}

\newcommand{\F}{\mathbb{F}}
\newcommand{\PP}{\mathbb{P}}

\newcommand{\C}{\mathcal{C}}

\newcommand{\LL}{\mathcal{L}}
\newcommand{\cO}{\mathcal{O}}

\newcommand{\cH}{\mathcal{H}}
\newcommand{\cE}{\mathcal{E}}
\newcommand{\CL}{\mathcal{C}_\mathcal{L}}

\def\fq{\mathbb{F}_{q}}
\def\fqs{\mathbb{F}_{q^2}}

\def\cF{\mathcal{F}}

\def\cX{\mathcal{X}}
\def\cC{\mathcal{C}}
\def\cF{\mathcal{F}}

\def\cH{\mathcal{H}}

\def\cX{\mathcal{X}}
\def\N{\mathbb{N}}
\def\Z{\mathbb{Z}}
\def\cD{\mathcal{D}}
\def\cF{\mathcal{F}}

\def\a{\alpha}

\DeclareMathOperator{\Ord}{ord}

\DeclareMathOperator{\lmd}{lmd}
\DeclareMathOperator{\Div}{Div}
\DeclareMathOperator{\Diff}{Diff}

\DeclareMathOperator{\Supp}{Supp}
\DeclareMathOperator{\Hull}{Hull}

\DeclareMathOperator{\Jac}{Jac}

\title[LCD and LCPs of AG Codes]{Linear Complementary Dual codes and Linear Complementary Pairs of AG codes in function fields}
\author{}
\date{}

\begin{document}
	\thanks{{\bf Keywords}: LCP of codes, LCD codes, Function fields, AG codes, algebraic curves}
	
	\thanks{{\bf Mathematics Subject Classification (2020)}:  94B05,  14G50, 11T71, 14H05}
	
	\thanks{The first author was partially supported by FAPEMIG: Grant APQ 00696-18 and Grant RED-0013-21. The second and third authors were financed in part by the Coordenação de Aperfeiçoamento de Pessoal de Nível Superior – Brasil (CAPES) – Finance Code 001. The third author was also partially financed  by 
		{\it Conselho Nacional de Desenvolvimento Científico e Tecnológico - Brasil }(CNPq) - 
		Project 307261/2023-9}
	
	\author{Alonso S. Castellanos, Adler V. Marques, and Luciane Quoos}
	
	\address{Instituto de Matemática e Estatística, Universidade Federal de Uberlândia, Santa Mônica, CEP 38408-100, Uberlândia, MG, Brazil}
	\email{alonso.castellanos@ufu.br}
	
	\address{Instituto de Matemática, Universidade Federal do Rio de Janeiro, Cidade Universitária,
		CEP 21941-909, Rio de Janeiro, RJ, Brazil}
	\email{adler@ufrj.br}
	
	\address{Instituto de Matemática, Universidade Federal do Rio de Janeiro, Cidade Universitária,
		CEP 21941-909, Rio de Janeiro, RJ, Brazil}
	\email{luciane@im.ufrj.br}

	\begin{abstract} 
		In recent years, linear complementary pairs  (LCPs) of codes and linear complementary dual (LCD) codes have gained significant attention due to their applications in coding theory and cryptography. In this work, we construct explicit LCPs of codes and LCD codes from function fields of genus $g \geq 1$. To accomplish this, we present pairs of suitable divisors that give rise to non-special divisors of degree $g-1$ in the function field. The results are applied in constructing LCPs of algebraic geometry codes and LCD algebraic geometry (AG) codes in Kummer extensions, hyperelliptic function fields, and elliptic curves.
	\end{abstract}
	
	\maketitle
	
	\section{Introduction}
		This paper focuses on the effective construction of families of Algebraic Geometry (AG) codes over algebraic function fields of genus $g\geq 1$, possessing the following desirable properties: Linear Complementary Pairs (LCPs) of codes and Linear Complementary Dual (LCD) codes. These families include some of the best-known error-correcting codes. Recently, Beelen, Rosenkilde, and Solomatov demonstrated that the family of AG codes supports an efficient list-decoding algorithm (see \cite{BRS2022}).
	
	Let $\fq$ be the finite field of order $q$, where $q$ is a power of a prime number. A linear code $\C$ over $\fq$ is an $\fq$-subspace of dimension $k$ in $\fq^n$, where $n \geq 1$. A code $\C$ is said to have length $n$, dimension $k$, and (Hamming distance) minimum distance $d = d(\C)$. These parameters are commonly denoted by $[n, k, d]$ or simply $[n, k]$ when convenient. For a given linear code $\C$, the dual code $\C^\perp$ is defined with respect to the Euclidean inner product in $\fq^n$ and has dimension $n - k$. The hull of $\C$, with parameters $[n, k]$, is defined as
$$
\Hull(\C) = \C \cap \C^\perp.
$$ 

The concept of the hull was first introduced by Assmus and Key in \cite{AK1990} for the classification of finite projective planes.
If $\dim(\Hull(\C)) = \ell$, it follows that $0 \leq \ell \leq \min{k, n-k}$. In this case, $\C$ is referred to as an $\ell$-dimensional hull code. Specifically, if $\ell = 0$, $\C$ is called an LCD code, and if $\ell = k$, the code $\C$ is said to be self-orthogonal, i.e., $\C^\perp \subseteq \C$.

	A pair $(\C, \cD)$ of linear codes over $\F_q$ of length $n$ is called an LCP of codes if $\C + \cD = \F_q^n$ and $\C \cap \cD = {0}$. In the special case where $\cD = \C^\perp$, the code $\C$ is an LCD code. LCD codes were first studied by Massey in 1964 \cite{Massey1964}, and in \cite{massey1992}, he demonstrated the existence of asymptotically good LCD codes. In 2004, Sendrier \cite{Sendrier2004} refined this result by proving that LCD codes achieving the asymptotic Gilbert-Varshamov bound exist.

Over the past decade, LCD codes and LCPs of codes have received increased attention due to their applications in cryptography. These include direct sum masking as a countermeasure against side-channel attacks (SCA) and fault injection attacks (FIA) \cite{bhasinetal2015, carlet2014, carlet-guilley2015}, as well as enhancing the robustness of encoded circuits against hardware Trojan horses \cite{ngo2014encoding, bhasinetal2015}.

In 2018, Carlet, Mesnager, Tang, Qi, and Pellikaan proved in \cite{CMTQP2018} that for $q > 3$, any linear code over $\fq$ is equivalent to an LCD code. This result implies that any set of parameters $[n, k, d]$ achievable by linear codes over $\fq$ with $q > 3$ is also achievable by Euclidean LCD codes. Jin and Xing \cite{JinXing2018} showed that an algebraic geometry code over $\mathbb{F}_{2^m}$ ($m \geq 7$) is equivalent to an LCD code. More recently, the classification of LCD codes has been extended to include codes with higher hull dimensions (see \cite{Luo-et-al}).

	Numerous studies have focused on the characterization and construction of LCPs of codes and LCD codes (see, for instance, \cite{massey1992, YangMassey1994, carletetal2018, CMTQ2018, SOK2018} and references therein). In the context of AG codes, Beelen and Jin \cite{BeelenJin2018} explicitly constructed several classes of MDS LCD codes over the rational function field. Mesnager, Tang, and Qi \cite{MTQ2018} investigated LCD AG codes and provided examples over the projective line, elliptic and hyperelliptic curves, and the Hermitian curve. A few years later, Bhowmick, Dalai, and Mesnager \cite{BDM2023} established conditions for constructing LCPs of codes and presented illustrative examples. Chen, Ling, and Liu \cite{ChenLingLiu2023} determined the dimension of the hull of Reed-Solomon codes using AG codes. In 2024, Moreno, López, and Matthews \cite{CLM2024} introduced an explicit construction of non-special divisors of "small" degree over Kummer extensions. These divisors have been effectively used to describe LCD algebraic geometry codes from certain families over Kummer extensions.

In this work, we focus on the effective construction of LCPs of codes and LCD codes over function fields of genus $g \geq 1$ over $\fq$. Our approach relies on the existence of non-special divisors of degree $g$ and $g-1$ supported on the rational places of the function field, as established in 2006 \cite{BL2006}. Theoretical conditions for non-special divisors to construct LCPs of codes over general function fields were provided in \cite[Theorem 3.5]{BDM2023} (see Theorem \ref{thmlcp3.5}). These constructions depend heavily on the existence of non-special divisors of degree $g-1$ in the function field. Moreover, for Kummer extensions, an explicit construction of such divisors was given in \cite[Theorems 8, 12, 13]{CLM2024} (see Theorem \ref{clmthm12}), leading to the construction of several LCD codes over specific function fields.
	
	In this work, we focus on the effective construction of LCPs of codes and LCD codes over function fields of genus $g \geq 1$ over $\fq$. Our approach relies on the existence of non-special divisors of degree $g$ and $g-1$ supported on the rational places of the function field, as established in 2006 \cite{BL2006}. Theoretical conditions for non-special divisors to construct LCPs of codes over general function fields were provided in \cite[Theorem 3.5]{BDM2023} (see Theorem \ref{thmlcp3.5}). These constructions depend heavily on the existence of non-special divisors of degree $g-1$ in the function field. Moreover, for Kummer extensions, an explicit construction of such divisors was given in \cite[Theorems 8, 12, 13]{CLM2024} (see Theorem \ref{clmthm12}), leading to the construction of several LCD codes over specific function fields.	
	
	The paper is organized as follows. In Section \ref{SectionPrelim}, we introduce the notation used in this paper. We also recall some necessary background about function fields, linear and algebraic geometry codes and some recent results on LCPs of codes and LCD  codes. In Section \ref{SectionKummer}, we propose effective construction of LCPs of AG codes and LCD AG codes over  Kummer extensions; see Theorems \ref{lcpkummer}, \ref{lcpkummer1} for LCPs of codes and Theorem \ref{lcpkummer2} for LCD codes. We conclude the section applying the obtained results for two families of maximal function fields, a generalisation of the Hermitian function field \cite{garcia1992}, and a function field covered by the Hermitian function field \cite{gatti2023galois}. In Section \ref{SectionHyper}, we explicitly construct LCPs and LCD codes from the function field of the hyperelliptic curve $y^{q+1}=x^2+x$ over $\F_{q^2}$ with $q$ odd, see Theorems \ref{lcphyper} and \ref{lcdx2+x}. Finally, in Section \ref{SectionElliptic}, we describe explicit LCPs of codes from elliptic curves that are MDS or near MDS; see Theorems \ref{LcpElliptic1} and \ref{LcpElliptic2}.

	\section{Preliminaries}
	\label{SectionPrelim}
	
	In this section, we briefly introduce function fields,  linear and algebraic geometry codes over finite fields and a background on LCD AG codes and LCPs of codes. 
	
	Throughout the paper, let $\fq$ be the finite field with cardinality $q = p^l$ for a prime $p$ and $\overline{\F}_q$ its algebraic closure.
	At first, we fix some arithmetical notations; we let $\lfloor \cdot \rfloor$ and $\lceil \cdot \rceil$ stand for the floor and ceil function respectively, and recall a useful equation about the floor function 
	\begin{equation}\label{floor}
		{\displaystyle \sum _{k=1}^{r-1}\left\lfloor {\frac {km}{r}}\right\rfloor ={\tfrac {1}{2}}(m-1)(r-1).}
	\end{equation}

	\subsection{Function fields}
	Let $\cF/\fq$ be an algebraic function field over $\F_q$ of genus $g=g(\cF)$ with field of constants $\fq$. We denote by $\PP_\cF$ the set of places of $\cF$. 
	A {\it divisor} $D$ is defined by 
	$$D=\sum_{P \in \PP_\cF} n_P P, \quad n_P \in \Z \text{ and } n_P=0 \text{ for almost all } P \in \PP_\cF.$$
	The {\it degree of a divisor} $D \in \Div(\cF)$ is defined by $\deg D = \sum_{P \in \PP_\cF} n_P \deg(P)$ and the {\it support of a divisor} $D$ by $\Supp(D) = \{ P \in \PP_\cF \mid n_P \neq 0 \}$. 
	To each place, $P \in \PP_\cF$, the corresponding discrete valuation associated with $P$ is denoted by $v_P$. For a function $z \in \cF$, we let $(z)$, $(z)_0$ and $(z)_\infty$ stand for the principal, zero and pole divisors of $z$, respectively. 
	
	Two divisors $D_1, \, D_2 \in \Div(\cF)$ are said to be {\it equivalent}, denoted by $D_1 \sim D_2$, if there is a function $z \in \cF$ such that $D_1-D_2=(z)$. 
	
	We let $\Div(\cF)$ be the free abelian additive group generated by the divisors in $\cF$. Given a divisor $G \in \Div(\cF)$, the Riemann-Roch space associated to $G$ is the finite-dimensional $\F_q$-vector space
	$$\LL(G) = \{ z \in \cF \mid (z) \ge - G\} \cup \{ 0 \},$$
	whose dimension is denoted by $\ell(G)$. 
	We notice the dimension of Riemann-Roch spaces remains unchanged under the equivalence relation of divisors. 
	
	Let $\Omega_\cF$ be the module of Weil differentials of $\cF$. For any nonzero $\omega \in \Omega_\cF$ we can associate a canonical divisor $(\omega)=\sum_{P \in \PP_\cF} v_P(\omega)P$.
	Given $x \in \cF$ a separating element, by \cite[Remark 4.3.7 (c)]{STICH2008}, the divisor of the differential $dx$ is 
	\begin{equation}\label{divisordif}
		(dx) = -2(x)_\infty + \Diff(\cF/\F_{q}(x)),
	\end{equation}
	where  $\Diff(\cF/\F_{q}(x))$ is the different of the extension field $\cF/\F_{q}(x)$.
	
	For a divisor $A \in \Div(\cF)$, we define the following submodule of  the module of Weil differentials 
		$$
		\Omega_{\mathcal{F}}(A)=\{ \omega \in \Omega_\mathcal{F} \mid \omega=0 \text{ or } (\omega) \ge A\},
		$$
		where  $(\omega)$ denotes the divisor of the differential. Denote the dimension of $\Omega_{\cF}(A)$ by $i(A)$, the index of speciality of $A$. It holds that $i(A)=\ell(W-A)$ for all canonical divisors $W$ of $\cF/\F_q$.
	
	The Riemann-Roch Theorem states that for any divisor $A$
	$$
	\ell(A)=\deg A + 1 - g + i(A).
	$$ 
	A divisor $A \in \Div(\cF)$ is called {\it non-special} if $i(A)=0$ and {\it special}, otherwise. If $deg(A) > 2g -2$, then $A$ is non-special. And if  $A$ is non-special, it has a degree of at least $g-1$. If $A$ has a degree exactly $g-1$, then $A$ is non-special if and only if $\ell(A)=0$. 
	
	Now, we establish a notation for places in a function field based on the algebraic curve associated with it, which will be employed in all subsequent results in the paper. Consider a function field $\cF/\fq$ defined by the non-singular absolutely irreducible curve $\cX$ defined  by the affine equation 
	$$\cX\,:\, f(x,y)=0,\mbox{ with }f(x,y)\in\mathbb{F}_q[x,y].$$ 
	For $a, b \in \overline{\F}_q$ such that $f(a,b)=0$, we denote the place $P_{ab}$ in $ \cF / \overline{\F}_q$ as the common zero of $x-a$ and $y-b$ in $\cF$. If $(a, b) \in \fq^2$, the place is rational.
	
	An important family of function fields suitable for applications in coding theory is the family of {\it maximal function fields}. These are function fields  $\cF/\fqs$ of genus $g$ attaining the famous Hasse-Weil upper bound for the number $N(\cF)$ of $\fqs$ rational points (places of degree one over $\fqs$)
	$$N(\cF)=q^2+1+2gq.$$

	\subsection{Linear and Algebraic Geometry Codes}
	A {\it linear code} $\C$ over $\F_q$ is an $\fq$-subspace of $\F_q^n$. Associated to a linear code, we have three parameters: the {\it length} $n$, the {\it dimension} $k$ of $\C$ as an $\F_q$-vector space, and the {\it minimum distance} (Hamming distance)
	$$d= \min \{ \text{ wt}(\textbf{c})  \mid \textbf{c} \in \C \setminus \{ 0\} \},$$ 
	where $\text{wt}(\textbf{x}) = \# \{ i \mid x_i \neq 0 \}$ for any  $n$-tuple $\textbf{x}=(x_1, \dots, x_n) \in \F_q^n$. We denote such code $\cC$ as an $[n,k,d]$ or just $[n,k]$ code.
	The simplest relation among the parameters of a code is the Singleton Bound $$k + d \le n+1.$$ 
	A code is said to be MDS if $n=k+d-1$.
	
	If $\C \subseteq \F_q^n$ is a linear code and $\cdot $ is the standard inner product in $\fq^n$,  the set 
	$$
	\C^\perp = \{ \textbf{x} \in \F_q^n \mid \textbf{x} \cdot \textbf{c}=0 \mbox{ for all } \textbf{c} \in \C \}
	$$
	is also a subspace of $\F_q^n$, it is called the {\it dual code} of $\C$.

	\begin{definition} Two codes $\C_1, \, \C_2 \subseteq \F_q^n$ are said to be equivalent if there is a vector $\textbf{a}=(a_1, \dots, a_n) \in (\F_q^*)^n$ such that
		$$
		\C_2 = \textbf{a} \cdot \C_1 = \{ (a_1c_1, \dots, a_n c_n) \mid (c_1, \dots, c_n) \in \C_1 \}.
		$$
	\end{definition}
	
	Observe that equivalent codes have the same dimension and minimum distance. 
	
	\begin{definition}
		A pair of linear codes $\C_1$ and $\C_2$ in $\fq^n$ is said to be a {\it linear complementary pair} (LCP for short) if the direct sum $\C_1 \oplus \C_2 = \F_q^n.$ 
	\end{definition}
	
	For an LCP of codes $(\C_1, \C_2)$, the minimum distance $d(\C_1)$ measures the protection against FIA, while $d(\C_2^\perp)$ measures the protection against SCA. The joint security against the two attacks is provided by  $\min\{d(\C_1), d(\C_2^\perp)\}$, which is called the security parameter for $(\C_1, \C_2)$. 
	
	\begin{definition}A  linear code $\C$  is said to be a {\it Linear Complementary Dual} (LCD for short) if $\C \oplus \C^\perp = \F_q^n.$ 
	\end{definition}
	
	We now recall the linear algebraic geometry (AG) codes considered in this work. Let $\cF/\F_q$ be a function field of genus $g$. Consider $P_1, \dots, P_n \in \PP_\cF$ pairwise distinct rational places of $\cF/\F_q$ and denote $D :=\sum_{i=1}^n P_i$. Given a divisor $G \in \Div(\cF)$ with $\Supp G \cap \Supp D = \emptyset$, the AG code $\C_\LL(D,G)$ is defined by 
	$$\CL(D,G) := \{ (f(P_1), \ldots , f(P_n)) \mid f \in \LL(G) \}. $$
	
	The parameters of AG codes satisfy the following well-known bounds.	
	\begin{proposition}\cite[Theorem 2.2.2]{STICH2008}\label{parameters} The code $\CL(D, G)$ has parameters $[n, k, d]$ which satisfy $$k=\ell(G)-\ell(G-D) \quad \mbox{and} \quad d \ge n - \deg G.$$ In particular, for $2g - 2 < \deg G < n$ it holds that $k=\ell(G)=\deg(G) + 1 - g$.  
	\end{proposition}

	The next proposition tells us that the dual code of a linear AG code is still a linear AG code and how it can be computed.
	\begin{proposition} \cite[Proposition 2.2.10]{STICH2008} 
		\label{dualAGcode}
		Let $\omega$ be a Weil differential such that $v_{P_i}(\omega)=-1$ and $\omega_{P_i}(1)=1$ for all $i=1, \dots, n$. Then
		$$\CL(D,G)^\perp  = \CL(D,H) \quad \mbox{with} \quad H:= D-G+(\omega).$$
	\end{proposition}
	
	Next, we observe that equivalent divisors give rise to equivalent AG codes.
	\begin{proposition}\cite[Proposition 2.2.14]{STICH2008}
		\label{EquivProp}
		Let $P_1, \dots, P_n$ be $n$ pairwise distinct rational places in $\cF/\F_q$, and consider $D=P_1 + \cdots + P_n$. 
		Consider $G$ and $H$ equivalent divisors such that $\Supp(D) \cap \Supp(G)= \Supp(D) \cap \Supp(H)= \emptyset$. Then $\CL(D,G)$ and $\CL(D,H)$ are equivalent. More specifically, if $G=H+(a)$ where $v_{P_i}(a)=0$ for all $i=1, \dots, n$, then  $$\CL(D, H)=\textbf{a} \cdot \CL(D,G),$$
		where $\textbf{a}=(a_1, \dots, a_n) \in (\F_q^*)^n$.
	\end{proposition}
	
	\subsection{LCPs of algebraic geometry codes}
	Let $\cF/\fq$ be an algebraic function field over $\F_q$ of genus $g$. For two divisors $A, B$ in $\Div(\cF)$, we define the {\it greatest common divisor} of $A$ and $B$ by
	$$\gcd(A,B) := \sum_{P \in \PP_\cF} \min\{ v_P(A), v_P(B) \} P,$$
	and  the {\it least multiple divisor} $A$ and $B$  by
	$$\lmd(A,B) := \sum_{P \in \PP_\cF} \max\{ v_P(A), v_P(B) \} P.$$
	The $\gcd(A,B)$ and $\lmd(A,B)$ fulfill these two fundamental properties
	\begin{equation}\label{basicP}
		\LL(A) \cap \LL(B) = \LL(\gcd(A,B)) \quad \mbox{and} \quad \gcd(A,B)+\lmd(A,B) = A+B.	
	\end{equation}
	
	Let $P_1, \dots, P_n$ be pairwise distinct rational places of $\cF$ and consider the divisor $D := \sum_{i=1}^n P_i$. Let $G$ and $H$ be two other divisors of $\cF$ such that 
	$$\Supp(G) \cap \Supp(D) = \Supp(H) \cap \Supp(D) = \emptyset.$$ 
	The pair $(\CL(D,G), \CL(D, H))$ of AG codes is an LCP of AG codes over $\F_q$ if 
	$$\CL(D,G) \oplus \CL(D, H) = \F_q^n.$$ 
	In other words, a pair $(\CL(D,G), \, \CL(D, H) )$ of AG codes is an LCP if and only if 
	\begin{equation}\label{LCPcond}
		\dim(\CL(D,G)) + \dim(\CL(D,H)) = n \quad \text{ and } \quad \CL(D,G) \cap \CL(D,H) = \{ 0 \}.
	\end{equation}
	
	In \cite{BDM2023}, Bhowmick, Dalai, and  Mesnager provided some adequate conditions on the divisors $G$ and $H$ to obtain a pair of LCP AG codes. 
	\begin{theorem}\cite[Theorem 3.5]{BDM2023}\label{thmlcp3.5}
		Let $\CL(D,G)$ and $\CL(D,H)$ be two algebraic geometry codes over a function field $\cF/\F_q$ of genus $g \neq 0$. Suppose $D$ has degree $n$ and the divisors $G$ and $H$ satisfy  
		\begin{enumerate}[i)]
			\item $2g-2 < \deg(G), \, \deg(H) < n$,  $\ell(G)+\ell(H) = n$,
			\item $\deg(\gcd(G, H))=g-1$, and 
			\item both divisors $\gcd(G,H)$ and $ \lmd(G,H)- D$ are non-special.
		\end{enumerate}
		Then the pair $(\CL(D,G), \, \CL(D,H))$ is LCP.
	\end{theorem}
	We notice that condition i) in Theorem	\ref{thmlcp3.5} is equivalent stating that $deg(G)+deg(H)=n+2g-2$ by Riemann-Roch Theorem. In most cases, it is simpler to calculate the degree of a divisor than the dimension of the Riemann-Roch space associated with it.
	
	Considering Theorem \ref{thmlcp3.5},  to derive the LCP of AG codes, one must identify non-special divisors of degree $g-1$. For this purpose, we employ the following theorem.
	\begin{theorem}\cite[Theorems 8, 12]{CLM2024}
		\label{clmthm12}
		Let $\cF/\F_q(x)$ be a Kummer extension given by $y^m = \prod_{i=1}^r (x-\alpha_i)$, where $\alpha_i \in \F_q$, $r < m$ and $\gcd(q, m)=\gcd(m, r) = 1$. Then, for $P_{\alpha_i0}$ the only zero of $x-\alpha_i$ in $\cF$, the divisor
		$$ A = \sum_{i=1}^{r-1} \floor*{ \frac{im}{r} } P_{\alpha_i0} - P$$	
		is  non-special  of degree $g-1$ for all $P \in \{ P_{ab} \mid b \neq 0 \mbox{ or } a \neq \alpha_i \} \cup \{ P_\infty \}$.
	\end{theorem}

	\section{LCD AG codes and LCPs of AG codes over Kummer extensions}	
	\label{SectionKummer}
	
	In this section, we consider  Kummer extensions $\cF/\F_q(x)$ defined by  
	\begin{equation}
		\label{kummerequation}
		y^m = f(x) = \prod_{i=1}^r (x- \alpha_i),
	\end{equation} 
	where   $\alpha_1, \dots, \alpha_r \in \F_q$ are pairwise distinct elements, $2 \leq r < m$ satisfies $\gcd(m, r) = 1$ and $m$ is a divisor of $q-1$. The function field $\cF/\F_q$ has genus $g=(r-1)(m-1)/2$. %For $a, b \in \overline{\F}_q$, we denote by $P_{ab}$ the common zero of $x-a$ and $y-b$ in $\cF$. 
	For each $1 \le i \le r$, let $Q_i = P_{\alpha_i 0}$ be the only rational place of $\cF$ associated to the zero of $x-\alpha_i$, and $Q_\infty$ as the unique place at infinity of $\cF$. Note that $\{ Q_1, \dots, Q_r, Q_\infty \}$ is the set of all totally ramified places in the function field extension $\cF/\F_q(x)$. Then we have the following divisors in $\cF$:
	\begin{enumerate}[i)]
		\item $(x- \alpha_i) = mQ_i - m Q_\infty$ for $1 \le i \le r$;
		\item $(y) = Q_1+ \cdots + Q_r - r Q_\infty$.
	\end{enumerate}

	The following theorem provides a method for constructing linearly complementary pairs of codes over such Kummer extensions.
	
	\begin{theorem}\label{lcpkummer}
		Let $\cF$ be the function field defined by Equation (\ref{kummerequation}) with at least $2g+r+1$ rational places. Suppose 
		$a_1, \dots, a_t$ in $ \fq$ are such that $P_{a_i}$ is a rational place in $\PP_{\fq(x)}$ completely split in the extension $\cF/\fq(x)$. Let  $$D=\sum_{i=1}^t P_{a_ib}|P_{a_i} \quad \text{  of degree } n.$$ For $s$ an integer with $m \le s < (n - g + 2)/r$, define the divisors
		\begin{align*}
			G &= \sum_{i=1}^{r-1} \floor*{ \frac{ im}{r} } Q_i + ( n - r s) Q_\infty, \text{ and}\\
			H &= \sum_{i=1}^{r-1} \left( s + \floor*{ \frac{(r-i)m}{r} } \right) Q_i + (s - 1)Q_r - Q_\infty.
		\end{align*}
		Then $\CL(D,G)$ is an $[n , n - r s  + 1]$-code, $\CL(D,H)$ is an $[n, r s - 1]$-code and the pair $(\CL(D,G), \, \CL(D,H))$ is an LCP of AG codes.
		Moreover, the security parameter of \newline $(\CL(D,G), \, \CL(D,H))$ is $d(\CL(D,G))$.
		
	\end{theorem}
	\begin{proof}
		We will show the divisors  $D, G$, and $H$ satisfy the conditions in Theorem \ref{thmlcp3.5}.
		At first, we notice that
		\begin{enumerate}[i)]
			\item $\Supp(G) \cap \Supp(D) = \Supp(H) \cap \Supp(D) = \emptyset$,
			\item $2g-2 < \deg(G)= g + n - rs, \, \deg(H)= rs + g - 2 < n$ (by Equation \ref{floor} and the condition $m \le s < (n - g + 2)/r$), 
			\item $\gcd(G,H) = \sum_{i=1}^{r-1} \floor*{ \frac{im}{r} }Q_i - Q_\infty$ since $\min \left\{ \floor*{\frac{ im}{r}}, s + \floor*{ \frac{(r-i)m}{r}} \right\}=\floor*{\frac{ im}{r}}$ for $s \ge m$, and 
			\item $\lmd(G,H) = \sum_{i=1}^{r-1} \left( s + \floor*{ \frac{(r-i)m}{r} } \right) Q_i + ( s - 1) Q_r + (n - rs)Q_\infty$.
		\end{enumerate}
		By the Riemann-Roch Theorem and Equations (\ref{floor}) and (\ref{basicP}), we obtain 
		\begin{align*} 
			\ell(G) + \ell(H) &  =  \deg(G + H) + 2 - 2g \\ 
			& = \deg(\gcd(G,H) + \lmd(G,H)) + 2 - 2g \\
			& = (g-1) + ((r-1)s + g + s - 1 + n - rs) + 2 - 2g \\
			&  =  n.
		\end{align*}
		Theorem \ref{clmthm12} gives $\gcd(G,H)$ is a non-special divisor of degree $g-1$. It is left to show that  $ \lmd(G,H)- D$ is also a non-special divisor. The divisor of the function $h = \prod_{i=1}^t (x-a_i)$ is $(h) = D- nQ_\infty$, and $(y^s)=s(Q_1+\cdots+Q_r)-rsQ_\infty$. This yields
		\begin{align*}
			\lmd(G,H) - D & = \lmd(G,H) - (h)-nQ_\infty\\
			& \sim \sum_{i=1}^{r-1} \left( s + \floor*{ \frac{(r-i)m}{r} } \right) Q_i + ( s - 1) Q_r + (n - rs )Q_\infty - nQ_\infty \\
			& = \sum_{i=1}^{r-1} \floor*{ \frac{(r-i)m}{r} } Q_i -  Q_r + \sum_{i=1}^r s Q_i - rsQ_\infty  \\
			& = \sum_{i=1}^{r-1} \floor*{ \frac{(r-i)m}{r} } Q_i -  Q_r + (y^s)  \\
			& \sim \sum_{i=1}^{r-1} \floor*{ \frac{(r-i)m}{r} } Q_i  - Q_r .
		\end{align*}
		We obtain that $\lmd(G,H) - D $ is equivalent to $\sum_{i=1}^{r-1} \floor*{ \frac{(r-i)m}{r} } Q_i  - Q_r$, a non-special divisor by Theorem \ref{clmthm12}. This concludes  $(\CL(D,G), \, \CL(D,H))$ is an LCP of codes by Theorem \ref{thmlcp3.5}.
		
		By Proposition \ref{dualAGcode}, the canonical divisor $W := \left( \frac{dh}{h} \right)$ is such that 
		$$\CL(D,G)^\perp = \CL(D, D + W - G).$$
		From $\Diff(F/\F_q(x))= \sum_{i=1}^r (m-1)Q_i + (m-1)Q_\infty$ and Equation (\ref{divisordif}), we obtain
		\begin{equation}
			W 
			=\left(\frac{dh}{dx}\right)+ (dx)-D + nQ_\infty  =-D +\left(\frac{dh}{dx}\right) +  \sum_{i=1}^r (m-1)Q_i + (n - m - 1)Q_\infty,
		\end{equation}				
		On the other hand, as $\floor*{\tfrac{im}{r}} + \floor*{\tfrac{(r-i)m}{r}} = m-1$, we have
		\begin{align*} 
			G+H & = \sum_{i=1}^{r-1} (s+m-1)Q_i + (s-1)Q_r + (n -rs - 1)Q_\infty \\
			& = \sum_{i=1}^r (s+m-1)Q_i - mQ_r + (n -rs - 1)Q_\infty \\
			& = \sum_{i=1}^r (m-1)Q_i + (n - m - 1)Q_\infty + \left(\frac{y^s}{x - \alpha_r} \right) \\
			& = D + W - \left(\frac{dh}{dx}\right) + \left(\frac{y^s}{x - \alpha_r} \right)\\
			&=D + W + \left(\frac{y^s}{h'(x - \alpha_r)} \right),
		\end{align*}
		where $h^\prime$ denotes the function $dh/dx$.
		Thus  
		$$\CL(D,G)^\perp = \CL\left(D, H  - \left(\tfrac{y^s}{h'(x - \alpha_r)} \right) \right)\sim \CL(D,H),$$ by Proposition \ref{EquivProp}. Therefore, the security parameter of $(\CL(D,G), \, \CL(D,H))$ is $d(\CL(D,G))$.
	\end{proof}
	
	In the following theorem, we show that an alternative choice of the divisors $G$ and $H$ in Theorem \ref{lcpkummer} also results in a pair of linear complementary codes.
	\begin{theorem}\label{lcpkummer1}
		Let $\cF$ be the function field defined in Equation (\ref{kummerequation}) with at least $2g+r+1$ rational places. Suppose 
		$a_1, \dots, a_t$ in $ \fq$ are such that $P_{a_i}$ is a rational place in $\fq(x)$ completely split in the extension $\cF/\fq(x)$. Let $$ D =\sum_{i=1}^t P_{a_ib}|P_{a_i} \quad \text{ of degree } n.$$ For $s$ an integer with $m \le s < \min\{ ( n - g + 1)/(r-1), (n + m - 1)/r \}$, define the divisors
		\begin{align*}
			G & = \sum_{i=1}^{r-1} \floor*{ \frac{im}{r} } Q_i + s Q_r + (n-rs-1)Q_\infty \\
			H & =  \sum_{i=1}^{r-1} \left( s + \floor*{ \frac{(r-i)m}{r} } \right) Q_i + (m-1)Q_r - m Q_\infty.
		\end{align*}
		Then  the pair $(\CL(D,G), \, \CL(D,H))$ is an LCP of AG codes with parameters $$[n , n - (r-1)s, \ge (r-1)s + g -1] \, \text{ and } \, [n, (r-1)s, n - (r-1)s + g -1],$$ respectively. 
		Moreover, the security parameter of $(\CL(D,G), \, \CL(D,H))$ is $d(\CL(D,G))$.
		
	\end{theorem}
	\begin{proof}
		The proof is very similar to the proof of Theorem \ref{lcpkummer}, so we prove the divisors  $\gcd(G,H)$ and $\lmd(G,H)-D$ are non-specials. We have
		\begin{align*}
			\gcd(G, H)&= \sum_{i=1}^{r-1} \floor*{ \frac{im}{r} } Q_i + (m-1) Q_r -mQ_\infty \\
			&= \sum_{i=1}^{r-1} \floor*{ \frac{im}{r} } Q_i -Q_r +(x-\alpha_r)\\
			&\sim \sum_{i=1}^{r-1} \floor*{ \frac{im}{r} } Q_i -Q_r.
		\end{align*}
		is a non-special divisor from Theorem \ref{clmthm12}.
		
		The divisor of the function $h = \prod_{i=1}^t (x-a_i)$ is $(h) = D- nQ_\infty$, and $(y^s)=s(Q_1+\cdots+Q_r)-rsQ_\infty$. This yields
		\begin{align*}
			\lmd(G,H)-D&= \sum_{i=1}^{r-1} \left( s + \floor*{ \frac{(r-i)m}{r} } \right) Q_i + sQ_r + (n-rs-1) Q_\infty-D\\
			&=\sum_{i=1}^{r-1} \left( \floor*{ \frac{(r-i)m}{r} } \right) Q_i + (n-1) Q_\infty +(y^s)-D\\
			&=\sum_{i=1}^{r-1} \left( \floor*{ \frac{(r-i)m}{r} } \right) Q_i   - Q_\infty +(y^s)-(h)\\
			& \sim \sum_{i=1}^{r-1} \left( \floor*{ \frac{(r-i)m}{r} } \right) Q_i   - Q_\infty.
		\end{align*}
		is also a non-special divisor from Theorem \ref{clmthm12}.
		Analogous to Theorem \ref{lcpkummer}, we can show that $\CL(D, G)^\perp = \CL\left(D, H - \left( \frac{y^s}{h'}\right)\right)$ and hence $\CL(D, G)$ is equivalent to $\CL(D, H)^\perp$.
		Hence the security parameter of $(\CL(D,G), \, \CL(D,H))$ is $d(\CL(D,G))$.
	\end{proof} 
	
	We notice that a class of codes satisfying that if $(\C, \cD)$ is LCP, then $\C$ is equivalent to $\cD^\perp$ is, for instance, the $\lambda$-constacyclic codes \cite{carletetal2018}.
	
	In the following theorem, we show that for certain maximal curves, the LCPs of codes obtained in Theorem \ref{lcpkummer1} can sometimes be a pair of LCD codes. We consider a maximal function field given by a Kummer extension  $y^{q+1}=f(x), f(x) \in \fqs[x]$ a separable polynomial with all its roots in $\fqs$. We point out that these two conditions are not really restrictive conditions for maximal function fields of Kummer type; see  \cite[Theorem 4.3]{ABQ2019} and \cite[Theorem 3.2]{TT2014}. In  \cite[Theorem 13]{CLM2024}, the authors gave necessary conditions on the divisors $G$ and $H$ such that a code over a maximal function field $y^{q+1}=f(x)$ is an LCD code. In what follows, we construct examples of such divisors.
	\begin{theorem}
		\label{lcpkummer2}
		Let $\cF/\F_{q^2}(x)$ be a maximal Kummer extension defined by 
		$$
		y^{q+1}=\prod_{i=1}^r (x - \alpha_i),
		$$ 
		where $\alpha_i \in \F_{q^2}$ for all $i=1, \dots, r$ and $2 \leq r < q+1$ coprime with $q+1$. Let $t\ge 2q+2$ be a divisor of $q^2-1$. Define the divisors 
		\begin{align*}
			D & =(y^{t}-1)_0,  \quad \mbox{ and }\\
			G & = \sum_{i=1}^{r-1} \floor*{ \frac{i(q+1)}{r} } Q_i + (t-q-1) Q_r + (r(q+1)-1)Q_\infty.
		\end{align*}
		Then $\CL(D,G)$ is an $[ rt, \, r(q+1)+t-q-1  , \, \ge (r-1)(t-q-1)-g+1]$-LCD code.
	\end{theorem}
	\begin{proof}
		As $\cF/\fqs$ is a maximal function field of genus $g=q(r-1)/2$, we have $N(\cF)=rq^2+1$ rational places in $\cF$. In particular, in the degree $r$ extension $\cF/\fqs(y)$, all rational places over $\fqs(y)$ except the only pole of $y$ are split. Thus, the divisor  $D=(y^{t}-1)_0$ has degree $\deg(D) = rt$.
		Let $$H =  \sum_{i=1}^{r-1} \left( (t - q- 1) + \floor*{ \frac{(r-i)(q+1)}{r} } \right) Q_i + q Q_r - (q + 1) Q_\infty.$$
		
		Now, following the notation as in  Theorem \ref{lcpkummer1} for $s=t-q-1$, we have the pair $(\CL(D,G), \, \CL(D,H))$ is LCP. We are left to prove $\CL(D,G)$ is an LCD code; we are going to show that $\CL(D,G) ^\perp= \CL(D,H).$ Notice that for any $1 \le i \le r-1,$  we have
		$$ 
		\floor*{ \frac{i(q+1)}{r} } + \floor*{ \frac{(r-i)(q+1)}{r} } = \floor*{ \frac{i(q+1)}{r} } + \floor*{ \frac{-i(q+1)}{r} } + q + 1 = q.
		$$
		Thus, by the choice of divisors $D, G$ and $H$, we have 
		$$H =  \sum_{i=1}^{r-1} \left( (t - q- 1) + \floor*{ \frac{(r-i)(q+1)}{r} } \right) Q_i + q Q_r - (q + 1) Q_\infty.$$
		So
		\begin{align*}  
			G + H - D & = \sum_{i=1}^{r-1} (t-1) Q_i + (t-1)Q_r + (r(q+1)-q-2) Q_\infty - (y^t-1)_0 \\
			& = \sum_{i=1}^{r} \left( t -1 \right) Q_i + (q r - q + r - 2 - rt)Q_\infty - (y^t-1)  \\
			& = (t - 1) \left( \sum_{i=1}^{r}  Q_i - rQ_\infty \right) + (q r - q - 2)Q_\infty - (y^t-1)  \\
			& = (y^{t-1}) + (2g-2)Q_\infty - (y^t-1) \\
			& = \left( \frac{d(y^t - 1)}{y^t - 1} \right) \quad \text{(by (\ref{divisordif}))},
		\end{align*}
		where $\left(\frac{d(y^t - 1)}{y^t - 1}\right)$ is a  canonical divisor with simple poles at $D$ and residue $1$ at all $P$ in $\Supp(D)$. By Proposition \ref{dualAGcode},  $\CL(D,G)^\perp = \CL(D,H)$ and $\CL(D,G)$ is an LCD code.
	\end{proof} 

\begin{remark}	
We notice that from \cite[Proposition 3]{carlet-guilley2015} an LCD code over $\mathbb{F}_{q^m}$ gives rise to an LCD code over $\fq$, if there exists a self-dual basis of $\mathbb{F}_{q^m}$ over $\mathbb{F}_q$. . In fact, given a $\cC$ an LCD code over $\mathbb{F}_{q^m}$, the expanded code relative to a self-dual basis of $\mathbb{F}_{q^m}$ over $\fq$ is also LCD. In particular, the construction of LCD codes over $\mathbb{F}_{2^m}$ provides LCD codes over $\mathbb{F}_2$.
\end{remark}	

	In the following subsections, we apply Theorems \ref{lcpkummer}, \ref{lcpkummer1}, and \ref{lcpkummer2} to construct LCPs of codes and LCD codes over certain function fields.

	\subsection{LCD AG codes and LCPs of AG codes over a generalization of the Hermitian function field}
	
	For $u \ge 1$ an odd integer, consider
	$\cF_u := \F_{q^{2u}}(x,y)$  the algebraic function field defined by  
	\begin{equation}
		\label{GHFequation}
		y^{q^u+1}=x^q+x
	\end{equation}
	over $\F_{q^{2u}}$. From  \cite[Example 1.3]{garcia1992} $\cF_u$ is a maximal function field of genus $g = \frac{q^u(q-1)}{2}$. The  $q^{2u+1}+1$ rational places in $\cF_u$  consist  of a unique place at infinity $Q_\infty$, and $q^{2u+1}$ rational places $P_{ab}$ with $a,b \in \F_{q^{2u}}$ satisfying $a^q+a=b^{q^u+1}$.

	\begin{proposition}\label{corherlcp}
		For $q^u+1 \le s \le \frac{2q^{2u+1}-q^{u+1}+q^u-2q+4}{2q}$ there exists a pair  LCP of AG codes with parameters
		$$[q^{2u+1} - q, q(q^{2u} - s - 1) + 1, \ge qs - q^u(q-1)/2  ]$$ and $$[q^{2u+1} - q, qs - 1, \ge q(q^{2u} - s -1) - q^u(q-1)/2 + 2 ].$$
	\end{proposition}
	\begin{proof}
		Let $\cF_u$ be the function field over $\F_{q^2}$ defined by Equation (\ref{GHFequation}) and the notation as above. Consider the divisor $D = \sum_{b \in \F_{q^{2u}}^*} P_{ab}$ of degree $q^{2u+1}-q $. For $q^u+1 \le s \le \frac{2q^{2u-1}-q^u+q^{u-1}-2}{2}$, consider the following  divisors 
		\begin{align*}
			G &= \sum_{i=1}^{q-1} i q^{u-1} Q_i + (q^{2u+1}-qs - q) Q_\infty, \text{ and } \\
			H &= \sum_{i=1}^{q-1} (s+ q^u-iq^{u-1})Q_i + (s-1)Q_q - Q_\infty.
		\end{align*}
		Then the pair $\left( \C_\LL(D, G), \, \C_\LL(D,H) \right)$ is an LCP of AG codes by Theorem \ref{lcpkummer}.
	\end{proof}	
	
	\begin{proposition}
		\label{PropHermitianlcd}
		Let $u \ge 1$ be an odd number. Then there exists an LCD code over $\F_{q^{2u}}$ with parameters 	
		$$
		\left[q^{2u+1} - q, q^{2u} + q^{u+1} - q^u + q - 2, \ge  (q-1)(2q^{2u}-q^{u+1}-q^u-2)/2 + 1 \right].
		$$
	\end{proposition}
	\begin{proof}
		Let $D = (y^{q^{2u}-1}-1)_0=\sum_{b \in \F_{q^{2u}}^*} P_{ab}$ be the sum of all affine rational places $P_{ab}$ with $b \neq 0$ of $\cF_u$. Then $D$ has degree $q^{2u+1}-q$. Consider
		\begin{align*}
			G &= \sum_{i=1}^{q-1} iq^{u-1} Q_i + (q^{2u}-q^u-2) Q_q + (q^{u+1}+q-1) Q_\infty.
		\end{align*}	
		Then $\C_\LL(D, G)$ is an LCD code over $\F_{q^{2u}}$ by Theorem \ref{lcpkummer2}.
	\end{proof}
	
	We now provide examples of LCPs of codes and an LCD code from Propositions  \ref{corherlcp} and \ref{PropHermitianlcd}, respectively.
	 \begin{example}
			Let $q=3$ and $\cF_1=\F_{3^2}(x,y)$ with $y^4=x^3+x$ be the Hermitian function field over $\F_{3^2}$ of genus $3$. For $4 \le s \le 7$, consider $D=\sum_{b \in \F_{9}^*} P_{ab}$ the sum of all rational places of $\cF_1$ with $b \neq 0$ of degree 24. For $Q_1, Q_2, Q_3$ the zeros of $x^3+x$ in $\cF_1$ let 
			$$G_s = Q_1 + 2Q_2 + (24-3s)Q_\infty \, \text{ and } \, H_s = (s+2)Q_1 + (s+1)Q_2 + (s-1)Q_3 - Q_\infty.$$
			From Proposition \ref{corherlcp}, $\left( \CL(D, G_s), \, \CL(D, H_s) \right)$ is an LCP of AG codes over $\F_9$ for each $4 \le s \le 7$. So, for $ s=4, \dots, 7$, we get LCPs of codes over $\F_9$ with parameters
			\begin{align*}
				&\left([24, 13,  9], \ [24, 11,  11] \right),  &\left([24, 10,  12], \ [24, 14,  8]\right),\\
				&\left([24, 7, 15], \ [24, 17,  5] \right),  &\left([24, 4,  18], \ [24, 21, 3] \right). 
			\end{align*}
		The minimum distance of the codes was computed using the free online version of Magma Computational Algebra System, see \cite{Magma}.
			Furthermore, by Proposition \ref{PropHermitianlcd}  for $G=Q_1+2Q_2+4Q_3+11Q_\infty$, we have $\CL(D, G)$ is an LCD code over $\F_9$ with parameters $[24, 16, 6]$ and $\CL(D, G)^\perp=\CL(D,H)$ with $H=6Q_1 + 5Q_2+3Q_3-4Q_\infty$.
		\end{example}

	\subsection{LCD AG codes and LCPs of AG codes over a maximal curve covered by the Hermitian function field}
	
	Let  $\F_{q^2}$ be a finite field of characteristic $p$. Recently, in a very nice paper \cite{gatti2023galois},  Gatti and Korchmáros provided explicit equations and classified all Galois subcovers of the Hermitian curve $\cH$ given as quotient curves by subgroups of the automorphism of  $\cH$ of order $p^2$. In this subsection, we describe certain LCPs of codes and LCD codes on one of the curves described in \cite{gatti2023galois}. 
	
	Let $q=p^h$ be a power prime with $h \ge 3$. Let $\mathcal{F}_b$ be the function field of the curve
	\begin{equation}
		\label{curveGattiKorch}
		\cX_b : \sum_{i=1}^{h-1} (b-b^{p^i}) x^{p^{i-1}} + w y^{q+1} = 0
	\end{equation}
	for $b \in \F_q \backslash \F_p$, where $w \in \F_{q^2}$ is a fixed element such that $w^{q-1}=-1$. 
	
	The function field $\mathcal{F}_b$ has genus $g(\mathcal{F}_b)=\frac{1}{2}q(p^{h-2} - 1)$ and it is a maximal function field over $\F_{q^2}$, that is, it has $N(\mathcal{F}_b)= p^{3h-2}+1$ places of degree one. Denote $Q_\infty$ the only place at infinity of $\mathcal{F}_b$ and $Q_1, \dots, Q_{p^{h-2}}$ the (rational) zeros of $$f(x)=-\frac{1}{w} \sum_{i=1}^{h-1} (b-b^{p^i}) x^{p^{i-1}}$$ in $\mathcal{F}_b$.
	
	\begin{proposition}
		\label{LCPGattiKorch}
		For $s$ an integer such that $ q+1 \le s \le q^2 - \frac{q-p^2}{2} - 1$, there exists a pair  $\left( \CL(D,G) , \, \CL(D, H) \right)$ of LCP of codes with parameters
		\begin{align*}
			&[p^{h-2}(q^2-1), \, p^{h-2}(q^2-s-1)+1, \, \ge \frac{p^{h-2}(2s - q) + q}{2}], \text{ and }\\
			&[p^{h-2}(q^2-1), \, p^{h-2}s-1, \, p^{h-2}(q^2-s- 1) - \frac{q(p^{h-2}-1)}{2} + 2],
		\end{align*}
		respectively.
	\end{proposition}
	\begin{proof}
		Consider the divisor  $D = \sum_{\beta \in \fqs^*} P_{\alpha \beta}$, where $P_{\alpha \beta}$ is a rational place on $\cF_b$ with $\beta \neq 0$ and $\sum_{i=1}^{h-1} (b-b^{p^i}) \alpha^{p^{i-1}} =- w \beta^{q+1} \neq 0 $. Then $D=(y^{q^2-1}-1)_0$ and has degree $\deg(D) = p^{h-2}(q^2- 1)$. Now, considering the divisors
		\begin{align*}
			G  &= \sum_{i=1}^{p^{h-2}- 1} p^2 i Q_i + p^{h-2}(q^2- s - 1) Q_\infty, \text{ and } \\
			H  &= \sum_{i=1}^{p^{h-2}- 1}  (s + q - p^2 i ) Q_i + (s-1)Q_{p^{h-2}} - Q_\infty,
		\end{align*}	
		the result follows from Theorem \ref{lcpkummer}.
	\end{proof}
	\begin{proposition}
		\label{lcdGattiK}
		Let $q=p^h$ be a prime power, where $h \ge 3$ is an integer.
		Let $1 \leq  v \leq (q-1)/2$ be a divisor of $q^2-1$. Then, there exists a
		$$
		\left[\tfrac{p^{h-2}(q^2-1)}{v}, \, p^{h-2}(q+1) + \tfrac{q^2-1}{v} - q - 1, \, \ge (p^{h-2}-1)(\tfrac{q^2-1}{v} - q - 1) - \tfrac{1}{2}q(p^{h-2} - 1) + 1 \right]
		$$
		LCD code over $\F_{q^2}$.
	\end{proposition}
	\begin{proof}
		Consider $\cF_b/\mathbb{F}_{q^{2u}}$ the function field over $\F_{q^2}$ defined by the curve (\ref{curveGattiKorch}). Let $t= \frac{q^2-1}{v}$. Since $q \ge 2v + 1$, we have $t > 2q + 1$. Now, we define 
		$$
		D = (y^t - 1)_0 \quad \mbox{and} \quad G = \sum_{i=1}^{p^{h-2}- 1} p^2 i Q_i + (t-q-1)Q_{p^{h-2}} + (p^{h-2}(q+1)-1)Q_\infty .
		$$
		From Theorem \ref{lcpkummer2}, we obtain that $\CL(D,G)$ is an LCD code.
	\end{proof}
	\begin{example} For $p=3$ and $h=3$, consider $\cF_b/ \F_{q^2}$ the function field over $\F_{q^2}$ of the curve (\ref{curveGattiKorch}). We have $q^2 - 1 = 27^2 -1 = 2^3 \cdot 5 \cdot 13$. In Proposition \ref{lcdGattiK}, we can take $v=1, 2, 4, 7, 8, 10, 13$. 
		
		For instance, for $v=4$ we have $t=(q^2-1)/4 = 182$ so we obtain a $[546, 238, \ge 282]$-LCD code over $\F_{27^2}$, namely $\CL(D,G)$ with 
		$$
		D=(y^{182} - 1)_0 \quad \mbox{and} \quad G=9Q_1 + 18 Q_2 + 154Q_3 + 83Q_\infty.
		$$
		
		For $v=7$, we have $t=8 \cdot 13=104$. So we obtain a $[312, 160, \ge 126]$-LCD code $\CL(D,G)$ over $\F_{27^2}$, where 
		$$
		D = (y^{104}- 1)_0 \quad \mbox{and} \quad G = 9Q_1 + 18Q_2 + 76Q_3 + 83 Q_\infty.
		$$
		Analogously, for $v=1,2,8, 13$, we get LCD codes with parameters 
		$$[2184, 784, \ge 1374], \ [1092, 420, \ge 646], \ [273, 147, \ge 100] \ \mbox{ and }  \ [168, 112, \ge 30],$$ respectively.
		
		For $p=2$ and $h=3$, we have $q^2 - 1 = 8^2 -1 = 3^2 \cdot 7$. In Proposition \ref{lcdGattiK}, choosing $v=3$ we have $t=(q^2-1)/3=21$ and obtain a $[42,30,\geq 9]$-LCD code over $\F_{64}$. And for $v=1$, we obtain a $[126,72,\geq 51]$-LCD code over $\F_{64}$, that is a code with the best-known parameters according \cite{MintTables}.
	\end{example}

		\begin{example}
			Let $p=2$, $h=3$ and $\zeta$ be a primitive $(2^6-1)$-th root of unity and $\F_{64}=\F_2(\zeta)$. Consider $b=\zeta^{q+1}$ and $w=\zeta^q+\zeta$ and  $\cF_b/ \F_{q^2}$ the function field over $\F_{q^2}$ of the curve (\ref{curveGattiKorch}). Let $D=(y^{63}-1)_0$ and $n=\deg(D) =2 \cdot 63 = 126$. For $9 \le s \le 61$, consider
			$$
			G_s =4Q_1+(126-2s)Q_\infty \quad \mbox{and} \quad H_s = (s+4)Q_1+(s-1)Q_2-Q_\infty.
			$$
			By Proposition \ref{LCPGattiKorch},
			$\left( \CL(D,G_s), \CL(D, H_s) \right)$ is an LCP of AG codes over $\F_{64}$ with parameters $[126, 127 - 2s, \ge 2(s-2)]$ and $[126, 2s - 1, \ge 124-2s]$, respectively. Thus, we can construct LCPs of codes over $\F_{64}$ so that the dimension  $k_s$ of $\CL(D,G_s)$ varies between all odd numbers in $5 \le k_s \le 109$. 
		\end{example}

	\section{LCPs of AG codes and LCD codes over the curve \texorpdfstring{$y^{q+1}=x^2+x$, $q$ odd}{yq+1=x2+x, q odd}}
	\label{SectionHyper}
	
	For $P$ a rational place in a function field $\cF/\fq$ of genus $g$, the {\it Weierstrass semigroup} at $P$ is
	$$H(P) = \{ n \in \N \mid \exists \,\, z \in \cF, \,\, (z)_\infty =nP\}.$$
	The elements in the semigroup are closely related to the dimension of certain Riemann-Roch spaces; in fact,  it is easy to notice that   $n \in H(P)$ if and only if $\ell((n-1)P)=\ell(nP)$. In particular, $0 \in H(P)$, and any integer larger than $2g-1$, is in the semigroup by the Riemann-Roch Theorem.
	
	This section considers a particular type of function field for which a rational place $P$ exists with Weierstrass semigroup $H(P)=\{0, g+1, g+2, \dots\}$. This is a  characteristic of hyperelliptic curves. In fact, in 1939, it was proved that for a rational place $P$ of a hyperelliptic function field that is not fixed by the hyperelliptic involution, the Weierstrass semigroup $H(P)=\{0, g+1, g+2, \dots\}$, see \cite[Satz 8]{S1939}. 
	
	Consider the function field $\cF = \F_{q^2}(x,y)/\fqs$ of the hyperelliptic curve defined by 
	\begin{equation}\label{hyp}
		\cX \,:\,y^{q+1}=x^2+x.
	\end{equation}
	
	If $q$ is even, this Kummer extension has a unique place at infinity $Q_\infty$ and was already considered in the previous section. Now, we focus on the scenario where $q$ is odd. The curve has genus $g=(q-1)/2$ in this case. 
	In \cite[Proposition 4.11]{TT2014}, Tafazolian and Torres  proved $\cX$ is a maximal curve over $\fqs$. The  $2q^2-q+1$ rational places can be characterized in the following way: 
	\begin{enumerate}
		\item  two places totally ramified, the roots of $x^2+x$ in $\cF$ denoted by $\{Q_0, Q_1\}$,
		\item two places at infinity $\{ Q_\infty, Q_\infty^{'} \}$, and
		\item $(2q-3)(q+1)$ rational places from $2q-3$ elements $\alpha \in \F_{q^2}^*$ such that $\alpha^2+\alpha$ is a $(q+1)$-power in $\F_{q^2}^*$. Observe that $0\neq \a^2+\a$ is a $(q+1)$-power in $\F_{q^2}^*$ if and only if $(\a^2+\a)^{q-1}=1$ if and only if $\a^{2q}+\a^q - \a^2 - \a = 0$ or, equivalently, if $\a$ is a root of $(x^q-x)(x^q+x+1)=0$. Since $x^2+x$ is a factor of $x^q-x$ and $2x+1$ is a factor of both $x^q-x$ and $x^q+x+1$, we obtain that the rational places $P_{\alpha \beta}$ in $\cF$ with $\beta \neq 0$ arise from the $2q-3$ elements $\a \in \F_{q^2}^*$ such that $\a$ is a root of 
		\begin{equation}
			\label{equationHyp}
			\frac{(x^{q-2}- x^{q-3}+ \cdots-x+1)(x^q+x+1)}{2x+1}.
		\end{equation}
		
	\end{enumerate}

	Consider the function $h :=\frac{ y^{q^2-1} - 1 }{ 2x + 1}$ in $\cF$. Defining $D$ as the zero divisor of $h$ we have 
	\begin{equation}\label{functionh}
		D := (h)_0= \sum_{b \in \F_{q^2}^*} P_{ab},
	\end{equation}  
	where the sum runs over all $a,b \in \F_{q^2}$ with $a^2+a=b^{q+1} \neq 0$. We also have  $\deg D = (2q-3)(q+1)$. 
		
	Now we construct non-special divisors of degree $g$ and $g-1$ on $\cF$.	
	
	\begin{lemma}
		\label{lemmanonspecyq+1=x2+x}
		Let $q$ be odd, and $\cF$ be the function field defined by Equation (\ref{hyp}). Let $P \in \{ P_{ab} \in \PP_\cF \mid 2a+1 \neq 0 \}$
		be a rational place of $\cF$. Then, $gP$ is a non-special divisor of degree $g$. In particular, $gP-P' \in \Div(\cF)$ is a non-special divisor of degree $g-1$ for all rational places $P' \in \PP_\cF$ distinct from $P$. 
	\end{lemma}
	\begin{proof}
		The hyperelliptic involution of $y^{q+1}=x^2+x$ is given by $\varphi(x,y)=(-x-1, y)$, so the fixed points of $\varphi$ are those $(a,b) \in \F_{q^2} \times \F_{q^2}$ such that $2a+1=0$. Hence for $P \in \{ P_{ab} \in \PP_\cF \mid 2a+1 \neq 0 \}$, we have $H(P) = \{0, g+1, g+2, ... \}$. Therefore $\LL(gP)=\{0\}$ and $gP$ is a non-special divisor of degree $g$.  Finally, given a rational place $P' \in \PP_\cF \setminus \{ P \}$, we have that $gP- P'$ is also a non-special divisor by \cite[Lemma 3]{BL2006}. 
	\end{proof}

	\begin{theorem}
		\label{lcphyper}
		Let $\cF=\F_{q^2}(x, y)$ be the function field defined by $y^{q+1}=x^2+x$ over $\F_{q^2}$. Let $Q_\infty$ and $Q_\infty'$ be the two places at infinity, and $Q_0$, $Q_1$ be the two totally ramified places in $\cF/\F_{q^2}(x)$. Let $D := (h)_0$ as in (\ref{functionh}), and $n := \deg(D) = 2q^2-q-3$. Suppose that $G, H \in \Div(\cF)$ are two divisors satisfying 
		\begin{enumerate}[i)]
			\item  $2g-2 < \deg(G), \, \deg(H) < n$;
			\item $\gcd(G, H) = g Q_0 - Q_\infty$;
			\item $\lmd(G, H) = g Q_0 + q Q_1 + (q^2-q-2)( Q_\infty + Q_\infty')$.
		\end{enumerate}
		Then the pair $\left( \C_\LL(D, G), \C_\LL(D,H) \right)$ is an LCP of AG codes.
	\end{theorem}
	\begin{proof}
		We shall show that $D$, $G$ and $H$ satisfy the conditions of Theorem \ref{thmlcp3.5}.
		By condition ii) and iii), we have $\deg(\gcd(G,H))=g-1$ and $\deg(\lmd(G,H))=2q^2-q+g-4$. From Riemann-Roch Theorem and Equation (\ref{basicP}),
		\begin{align*}
			\ell(G) + \ell(H) 
			& = \deg(G+H) + 2 - 2 g \\
			& = \deg( \gcd(G,H) + \lmd(G,H) ) + 2 - 2g \\
			& = 2q^2-q-3=n.
		\end{align*}
		Lemma \ref{lemmanonspecyq+1=x2+x} yields that $\gcd(G,H)=g Q_0 - Q_\infty$ is a non-special divisor of degree $g-1$. We are left to show that $\lmd(G,H)-D$ is also a non-special divisor. We have $D$ is the zero divisor of the function $h=\frac{y^{q^2-1}-1}{2x+1}$, and the principal divisor of $h$ is 
		$(h) = D - \frac{2q^2-q-3}{2}(Q_\infty+Q_\infty')$. Furthermore, we have the divisor $(y)=Q_0+Q_1-Q_\infty-Q_\infty'$ and, after some computations, we conclude 
		$$\lmd(G,H) - D \sim - Q_0 + g Q_1,$$
		which is a non-special divisor of degree $g-1$ by Lemma \ref{lemmanonspecyq+1=x2+x}. Therefore, by Theorem \ref{thmlcp3.5}, the pair $\left( \C_\LL(D,G), \,  \C_\LL(D, H) \right)$ is an LCP of AG codes.
	\end{proof}
	\begin{corollary}
		\label{CorHyperLCP1}
		Let $\cF$ and $D$ be as in Theorem \ref{lcphyper}. Consider the divisors
		$$
		G = gQ_0 + qQ_1 + (q^2-q-2)Q_\infty
		$$
		and
		$$
		H = gQ_0 - Q_\infty + (q^2-q-2)Q_\infty'.
		$$
		Then $(\CL(D,G), \CL(D,H))$ is an LCP of AG codes with parameters 
		\begin{align*}
			&[2q^2-q-3, q^2-1, \ge q^2-q-1 - \tfrac{q-1}{2}]  \mbox{ and }\\
			&[2q^2-q-3, q^2-q-2,  \ge  q^2-\tfrac{q-1}{2}],
		\end{align*}
		respectively.
	\end{corollary}

	We finish this section giving conditions to an LCP of AG codes $\left( \C_\LL(D, G), \, \C_\LL(D,H) \right)$  to be such that $\C_\LL(D, G)^\perp = \C_\LL(D, H)$, i.e., $\C_\LL(D, G)$ is an LCD code.
	
	\begin{lemma}
		\label{diffh}
		Let $h = \frac{ y^{q^2-1} - 1 }{ 2x + 1 } $ a rational function in $\cF$. Then we have
		$$
		\left( \frac{dh}{h} \right) = -D + q(Q_0+Q_1) + (q^2-q-3)(Q_\infty + Q_\infty') + \left( y^{q^2-q -2} +\frac{2h}{2x+1} \right).
		$$
	\end{lemma}
	\begin{proof}
		Since $y^{q+1}=x^2+x$, the product rule yields $y^q dy = (2x +1)dx$, that is, $dy= \frac{2x+1}{y^q} dx$. Hence, we obtain that    
		\begin{align*}
			dh & = d\left( \frac{ y^{q^2-1} - 1 }{ 2x + 1 }  \right)\\
			& =\frac{ -y^{q^2-2}(2x+1) dy - 2(y^{q^2-1}-1)dx }{ (2x + 1)^2 } \\
			& =  \frac{-y^{q^2-q-2}(2x+1)^2 - 2(y^{q^2-1}-1)}{(2x+1)^2} dx \\
			& = - \left( y^{q^2-q -2}+ \frac{2h}{2x+1} \right) dx.
		\end{align*}
		This yields,
		\begin{align*}
			\left(\frac{dh}{h} \right) & = -(h) +  \left( y^{q^2-q -2} + \frac{2h}{2x+1} \right)  + (dx) \\
			& = -D + \frac{2q^2-q-3}{2} (Q_\infty + Q_\infty') + \left( y^{q^2-q -2} + \frac{2h}{2x+1} \right) \\
			& \hspace*{6cm} + q(Q_0 + Q_1) - \frac{q+3}{2}(Q_\infty + Q_\infty') \\
			& = - D + q(Q_0 + Q_1) + (q^2-q-3)(Q_\infty + Q_\infty') + \left( y^{q^2-q -2} + \frac{2h}{2x+1} \right).
		\end{align*}
	\end{proof}

	\begin{theorem}
		\label{lcdx2+x}
		Let $\cF=\F_{q^2}(x,y)$ be the function field over $\F_{q^2}$ defined by $y^{q+1}=x^2+x$. Consider $D=(h)_0$.
		If $G, H \in \Div(\cF)$ are divisors of $\cF$ such that 
		\begin{enumerate}[i)]
			\item $2g- 2 < \deg(G), \, \deg(H) < \deg(D)$;
			\item $\gcd(G,H) = gQ_0 - Q_\infty$; and
			\item $G+H = q(Q_0+Q_1)+(q^2-q-3)(Q_\infty + Q_\infty')$.
		\end{enumerate}
		Then the pair $\left( \C_\LL(D, G), \C_\LL(D,H) \right)$ is an LCP of AG codes.	Moreover, the code $\CL(D,G)$ is equivalent to an LCD code.
	\end{theorem}
	\begin{proof}
		Suppose $G, H \in \Div(\cF)$ are divisors of $\cF$ satisfying $(i)$, $(ii)$ and $(iii)$. By Proposition \ref{dualAGcode} and Lemma \ref{diffh}, we have $\CL(D,G)^\perp = \CL(D, H')$ where
		$$
		H' := D-G+\left( \frac{dh}{h} \right)=-G + q(Q_0 + Q_1) + (q^2-q-3)(Q_\infty + Q_\infty') + \left( y^{q^2-q -2} + \frac{2h}{2x+1} \right).
		$$
		Let $z = y^{q^2-q -2} + \frac{2h}{2x+1} \in \cF$. Note that $z(P_{ab}) \neq 0$ for all $P_{ab} \in \Supp(D)$.
		
		By assumption $(iii)$, $H = H' - \left( z \right)$.   
		Since $\gcd(G,H) = gQ_0 - Q_\infty$, we have 
		\begin{align*}
			\lmd(G,H) - D & = G+H-\gcd(G,H) - D  \\
			& = \tfrac{q+1}{2} Q_0 + qQ_1 + (q^2-q-2)Q_\infty + (q^2-q-3)Q_\infty' - (h)_0 \\
			& \sim \tfrac{q+1}{2} Q_0 + qQ_1 + (q^2-q-2)Q_\infty + (q^2-q-3)Q_\infty' \\
			& \hspace*{6.2cm}- \tfrac{2q^2-q-3}{2}(Q_\infty + Q_\infty') \\
			& = \tfrac{q+1}{2} Q_0 + qQ_1 - \tfrac{q+1}{2}Q_\infty -  \tfrac{q+3}{2}Q_\infty' \\
			& = \tfrac{q-1}{2} Q_1 - Q_\infty' + ( y^{(q+1)/2}) \\
			& \sim gQ_1 - Q_\infty',
		\end{align*}
		that is, $\lmd(G,H)-D$ is equivalent to $gQ_1 - Q_\infty'$.
		Hence $\gcd(G,H) = gQ_0 - Q_\infty$ and $\lmd(G,H)-D$ are non-special divisors of degree $g-1$, by Lemma \ref{lemmanonspecyq+1=x2+x}. Therefore,
		the pair $\left( \C_\LL(D,G), \,  \C_\LL(D, H) \right)$ is an LCP of AG codes. 
		
		Now we show $\CL(D,G)$ is equivalent to an LCD code. By Equation (\ref{equationHyp}), we have 
		$$ h = \frac{ y^{q^2-1} - 1 }{ 2x + 1 }=\frac{(x^{q-2}- x^{q-3}+x^{q-4}- \cdots+x-1)(x^q+x+1)}{2x+1} \quad \mbox{in } \cF, $$
		and we get 
		$$  \frac{h}{2x+1} =  \tilde{h}(x)  \in \fq[x], $$
		For $P_{ab} \in \Supp(D)$ and $2a+1 \neq 0$, we have 
		$$
		z(P_{ab}) = b^{q^2-q-2}+\frac{h(a,b)}{2a+1} = (b^{-1})^{q+1} =\text{N}_{\F_{q^2}/\F_q}(b^{-1}) \in \F_{q};
		$$
		and for $P_{ab} \in \Supp(D)$ with $2a+1=0$, we have
		$$
		z(P_{ab}) = b^{q^2-q-2}+\tilde{h}(a) = (b^{-1})^{q+1} + \tilde{h}(a) = \text{N}_{\F_{q^2}/\F_q}(b^{-1}) + \tilde{h}(a) \in \F_{q}.
		$$
		As $z(P_{ab}) \neq 0$ and $z(P_{ab}) \in \F_{q}$ for all $P_{ab} \in \Supp(D)$, we obtain that $z(P_{ab})$ is a nonzero square in $\F_{q^2}$ for all $P_{ab} \in \Supp(D)$. 
		Thus, for $w \in \cF$ a function such that  $w(P_i)^2=z(P_i)$ for all $1 \le i \le n$ and as $(dh/h)=(G+(w)) + (H+(w)) -D$,  by Proposition \ref{dualAGcode} we obtain
		$$ \CL(D, G +(w) )^\perp= \CL(D, H + (w)). $$
		Hence 
		$$\CL(D, G + (w)) \cap \CL(D, H + (w)) = \textbf{w} \cdot (\CL(D, G) \cap \CL(D, H)) = \{0\},$$ 
		where $\textbf{w}=(w(P_1), \dots, w(P_n))$. Therefore, $\CL(D, G+(w))$ is an LCD code, and the theorem follows.
	\end{proof}
	\begin{corollary}
		\label{CorCodigos}
		Let $q$ be odd. Then there are LCD codes over $\F_{q^2}$ with parameters
		\begin{align*}
			& [2q^2-q-3, q^2-3, \ge q^2-q+1-\tfrac{q-1}{2}],\\
			& [2q^2-q-3, q^2-q-3, \ge q^2+1-\tfrac{q-1}{2}], \\
			& [2q^2-q-3, q, \ge 2q^2-2q-2-\tfrac{q-1}{2}] \quad \mbox{and} \\
			& [2q^2-q-3, q^2-q-1, q^2-1-\tfrac{q-1}{2}]
		\end{align*}
	\end{corollary}
	\begin{proof}
		For $1 \le i \le 4$, we define $\C_i := \CL(D, G_i)$ where
		\begin{align*}
			G_1 & = g Q_0 + q Q_1 - Q_\infty + (q^2-q-3) Q_\infty', \\
			G_2 & = g Q_0  - Q_\infty + (q^2 - q- 3)Q_\infty', \\
			G_3 & = g Q_0 + q Q_1 - Q_\infty \\
			G_4 & = g Q_0 + (q^2-q-2)Q_\infty.
		\end{align*}
		Let $G = G_1$. We have $$ \deg(G) = \frac{q-1}{2}+q-1+q^2-q-3 = \frac{2q^2+q-9}{2},$$
		thus $$q-3 = 2g-2 < \deg(G) < 2q^2-q-3 = \deg(D).$$ 
		Take $H=(g+1)Q_0 + (q^2-q-2)Q_\infty$, and hence $2g-2 < \deg(H) < \deg(D)$. Moreover
		$$ \gcd(G,H) = g Q_0 - Q_\infty$$
		and
		$$ G+H = \gcd(G,H) + \lmd(G,H) = q(Q_0+Q_1)+(q^2-q-3)(Q_\infty + Q_\infty').$$
		Therefore, by the Theorem \ref{lcdx2+x}, $\CL(D,G)=\CL(D,G_1)$ is equivalent to a $[2q^2-q-3, q^2-3, \ge q^2-q+1 - \tfrac{q-1}{2}]$-LCD code over $\F_{q^2}$. The proof for the other $G_i$'s follows analogously, and we omit the details.
	\end{proof}

		\begin{example}
			Consider the function field $\cF=\F_{5^2}(x,y)$ with $y^6=x^2+x$ over $\F_{5^2}$ of genus $g=(5-1)/2=2$. For $D=(h)_0$ with $h=(y^{24}-1)/(2x+1)$, we have $n=\deg D= 42$. Further, for $Q_0, Q_1$ the zeros of $x^2+x$ in $\cF$, consider
			\begin{align*}
				G_1 = 2Q_0 + 5Q_1 + 18Q_\infty \quad & \mbox{and} \quad H_1 =2Q_0 - Q_\infty + 18Q_\infty' \end{align*}
			
			Then, by Corollary \ref{CorHyperLCP1}, $\left( \CL(D, G_1), \CL(D, H_1) \right)$ is an LCP of codes over $\F_{5^2}$ with parameters $[42, 24,  17]$ and $[42, 18,  23]$, respectively.
			Now, considering
			\begin{align*}
				G_2 = 2Q_0 + 5Q_1 - Q_\infty + 17Q_\infty' \quad & \mbox{and} \quad H_2 =3Q_0 + 18Q_\infty,
			\end{align*}
			by Theorem \ref{lcdx2+x} and Corollary \ref{CorCodigos},
			$\left( \CL(D, G_2), \CL(D, H_2) \right)$ is an LCP of codes over $\F_{5^2}$ with parameters $[42, 22,  19]$ and $[42, 20,  21]$, respectively. Furthermore, we have
			$\CL(D, G_2)^\perp=\textbf{z}\cdot\CL(D, H_2)$ where $\textbf{z}=(z(P_1), \dots, z(P_{42}))$ and $z=y^{18}+2\frac{y^{24}-1}{(2x+1)^2}$. The minimum distance of the codes was computed using the free online version of Magma Computational Algebra System, see \cite{Magma}.
		\end{example}

	\section{LCPs of AG codes over elliptic function fields}
	\label{SectionElliptic}
	
	In this section, we present the construction of LCPs of codes over elliptic curves. Let $\cE/\fq$ be an elliptic function field over $\F_q$, and $\cO$ be the only place at infinity of $\cE(\overline{\F}_q)$. It is well known that there is an isomorphism between $\Jac(\cE)$, the Jacobian of $\cE$, and $\cE(\F_q)$,  the set of rational places of $\cE/\F_q$. In particular, $\cE(\F_q)$ can be viewed as the abelian group $(\Jac(\cE) , \oplus)$ with $\cO$ as the neutral element (see \cite[Section 6.1]{STICH2008} or \cite{silverman2009}). It is also well known that the  group $\Jac(\cE)$ has the form $\Z/m\Z \times \Z/mk \Z$ for some positive integers $m$ and $k$. For a description of all possible groups, see \cite{ruck1987,Voloch1988}.
	
	Let $P_1, \dots, P_n \in \Jac(\cE)$ be $n$ distinct rational places, and let $D=P_1 + \dots + P_n$. For a divisor $G \in \Div(\cE)$ such that $0 < \deg(G) < n$ and $\Supp(G) \cap \Supp(D) = \emptyset$, by  Proposition \ref{parameters}, the elliptic code $\CL(D,G)$ is an $[n, k, d]$-code where 
	\begin{equation}
		\label{ellcodesparameters}
		k=\deg(G) \quad \mbox{and} \quad d=n-k \mbox{ or } d=n-k+1.
	\end{equation}
	Moreover, $\CL(D,G)$ is not MDS if and only if there are $k$ distinct places $P_{i_1}, \dots, P_{i_k} \in \Supp(D)$ such that $G \sim P_{i_1}+ \cdots + P_{i_k}$ or, equivalently, $G = P_{i_1} \oplus \cdots \oplus P_{i_k}$ in $\Jac(\cE)$.
	
	In an elliptic function field, we observe a divisor $A \in \Div(\cE)$ of degree zero is non-special if and only if $A$ is not principal. Next, we describe some LCPs of elliptic codes. 
	
	\begin{theorem}\label{LcpElliptic1}
		Let $\cE$ be the function field of an elliptic curve over $\F_q$.
		Suppose the number $\#\Jac(\cE)$ of $\fq$-rational places in $\cE$  is odd or $\#\Jac(\cE)$ is even and $\Jac(\cE)$ contains more than one element of order $2$. Fix $P_0 \in \Jac(\cE)$ a rational place of $\cE$ distinct from $\cO$ and denote $\alpha =\Ord(P_0) > 1$ its order in $\Jac(\cE)$. Let $D := \sum_{P \in \Jac(\cE) \setminus \{P_0, \cO\}} P$ and $n := \deg (D)$. For any two positive integers $a,b$ such that 
		$$ 
		0 < b - a < n, \quad a \not\equiv 0 \pmod{\alpha} \quad \mbox{and} \quad b \not\equiv -1 \pmod{\alpha},
		$$
		we have that $$(\C_\LL(D, a P_0 + (n-b) \cO), \, \C_\LL(D, b P_0 - a \cO ))$$ is an LCP of elliptic codes over $\F_q$.
	\end{theorem}
	\begin{proof}
		It is enough to prove that the divisors $G=aP_0 + (n-b) \cO$ and $H=bP_0 - a \cO$ satisfy the hypothesis in Theorem \ref{thmlcp3.5}. From $0 < a < b < n+a$, we have 
		$
		0 < \deg(G), \, \deg(H) < n, \, \gcd(G,H)  = a P_0 - a \cO$ and $ \lmd(G,H) = b P_0 + (n-b) \cO.$
		Clearly,
		$\dim(\CL(D,G)) + \dim(\CL(D,H)) = \deg(G) + \deg(H) = n.$
		
		As $\Jac(\cE)$ is an abelian group with $\#\Jac(\cE)$ being odd or $\Jac(\cE)$ containing more than one element of order $2$, we have
		\begin{equation}
			\label{ellsum01}
			\sum_{P \in \Jac(\cE)} P = \cO \quad \mbox{in } \Jac(\cE).
		\end{equation}
		Moreover, $\lmd(G,H)-D$ is a principal divisor if and only if $\lmd(G,H) \circleddash D = \cO$ in $\Jac(\cE)$,  where, for $P \in \Jac(\cE)$, $\circleddash P$ denotes the inverse of $P$ in $\Jac(\cE)$ with respect to $\oplus$. On the other hand, by equation $(\ref{ellsum01})$ we have
		$$\lmd(G,H) \circleddash D  = bP_0 \circleddash D 
		= (b+1)P_0 \circleddash \sum_{P \in \Jac(\cE)} P 
		= (b+1)P_0\neq \cO,$$
		since $b +1 \not\equiv 0 \pmod{\alpha}$. Since $\deg(\lmd(G,H)-D)=0$ we conclude $\lmd(G,H)-D$ is non-special.
		
		Now,  since $\gcd(G,H)=a(P_0-\cO)$ and $a \not\equiv 0 \pmod{\alpha}$, we have $a(P_0-\cO) \neq \cO$ in $\Jac(\cE)$, that is, $\gcd(G,H)$ is also a non-special divisor. This concludes the proof.
	\end{proof}
	\begin{corollary}
		\label{Corell1} Let $\cE$ and $D$ be as in Theorem \ref{LcpElliptic1}.
		Suppose $\Jac(\cE)$ is a cyclic group of odd order $\alpha \ge 5$. Let $P_0$ be a generator of $\Jac(\cE)$. Then for any positive integers $a$ and $b$ with $a < b < \alpha - 1$, we have 
		$$(\C_\LL(D, a P_0 + (\alpha-b-2) \cO), \, \C_\LL(D, b P_0 - a \cO ))$$ is an LCP of elliptic codes over $\F_q$.
	\end{corollary}
	\begin{proof}
		Note that $a < b< \alpha-1$ implies $a \not\equiv 0 \pmod{\alpha}$ and $b \not\equiv -1 \pmod{\alpha}$.
		Now apply Theorem \ref{LcpElliptic1} with $n=\alpha-2$. 
	\end{proof}
	\begin{corollary}
		\label{Corell2}
		Let $\F_q$ be such that $\emph{char}(\F_q)>2$ and $\cE$ be the function field of $y^2=(x-e_1)(x-e_2)(x-e_3)$ over $\F_q$, where $e_1, e_2, e_3 \in \F_q$ are pairwise distinct.   Define $D := \sum_{P \in \Jac(\cE) \setminus \{P_0, \cO\}} P$ of degree $n$. For any $a, b$ positive integers such that 
		$$
		0 < b-a < n, \quad a \mbox{ odd} \quad  \mbox{and} \quad b  \mbox{ even}, 
		$$ and  $P_0$ in $\{ P_{e_1 0}, P_{e_2 0}, P_{e_3 0} \}$ we have
		$$ 
		\left( \CL(D, aP_0 + (n-b)\cO), \, \CL(D, bP_0 - a\cO) \right) 
		$$ 
		is an LCP of elliptic codes.
	\end{corollary}
	\begin{proof}
		The $\Jac(\cE)$ contains more than one element of order $2$, namely rational places  $P_{e_1 0}$, $P_{e_2 0}$ and  $P_{e_3 0}$. Moreover, since $a \equiv 1 \pmod{2}$ and $b \equiv 0 \pmod{2}$, the result follows from Theorem \ref{LcpElliptic1}. 
	\end{proof}
	
	In Theorem \ref{LcpElliptic1}, we use the fact that $\Jac(\cE)$ contains more than one element of order $2$ or has order odd. Next, we consider the case when $\#\Jac(\cE)$ is even and $\Jac(\cE)$ contains exactly one element of order $2$. 
	\begin{theorem}
		\label{LcpElliptic2}
		Let $\cE$ be the function field of an elliptic curve over $\F_q$.
		Suppose that $\Jac(\cE)$ contains exactly one element of order $2$. Fix $P_0 \in \Jac(\cE) \setminus \{ \cO \}$ with $\alpha := \Ord(P_0)$.
		Let $D = \sum_{P \in \Jac(\cE) \setminus \{P_0, \cO\}} P$ and $n = \deg (D)$. Let $a$, $b$ be two positive integers such that $0 < b - a < n$ and  $a \not\equiv 0 \pmod{ \alpha}$. If
		\begin{enumerate}[i)]
			\item $\alpha$ is odd; or
			\item $\alpha$ is even with $\gcd(\alpha, b+1) \neq \frac{\alpha}{2}$,
		\end{enumerate} 
		then 
		$$
		(\C_\LL(D, a P_0 + (n-b) \cO), \, \C_\LL(D, b P_0 - a \cO ))
		$$
		is an LCP of elliptic codes.
		
	\end{theorem}
	\begin{proof}	We will show that $D$, $G=aP_0 + (n-b) \cO$ and $H=bP_0 - a \cO$ satisfy the conditions in Theorem \ref{thmlcp3.5}. Observe that $\ell(G)+\ell(H)=n$. Since $\Jac(\cE)$ contains exactly one element of order $2$, 
		\begin{equation}
			\label{ellsum02}
			\sum_{P \in \Jac(\cE)} P = \tilde{P} \quad \mbox{in} \quad \Jac(\cE),
		\end{equation}
		where $\tilde{P}$ is the only element in $\Jac(\cE)$ of order $2$. Adding on $\Jac(\cE)$ and using (\ref{ellsum02}), we obtain
		$$
		\lmd(G,H) \circleddash D  = (b+1)P_0 \circleddash \tilde{P},
		$$ 
		where the last equality follows from Equation (\ref{ellsum02}).
		Hence $\lmd(G,H) \circleddash D = \cO$ on $\Jac(\cE)$ if and only if $(b+1)P_0$ has order $2$ in $\Jac(\cE)$. Note that if $\alpha$ is odd, then $\tilde{P} \notin \langle P_0 \rangle$. Thus $(b+1)P_0 \neq \tilde{P}$ for any $b$. Suppose $\alpha$ is even. Then $\tilde{P} \in \langle P_0 \rangle$ as $\Jac(\cE)$ contains only one element of order $2$, so
		$$
		\Ord((b+1)P_0) = \frac{ \alpha}{\gcd(\alpha, b+1)}.
		$$ 
		If $\alpha$ is odd or $\alpha$ is even with $\gcd(\alpha, b+1) \neq \frac{\alpha}{2}$, then  
		$$
		\gcd(G,H)=a(P_0 \circleddash \cO) \neq \cO \quad \mbox{and} \quad \lmd(G,H) \circleddash D \neq \cO \ \mbox{ in } \ \Jac(\cE),
		$$
		that is, $\gcd(G,H)$ and $\lmd(G,H)-D$ are both non-special divisors. This concludes the proof.
	\end{proof}
	\begin{corollary}
		\label{Corell3}
		With  assumptions as in Theorem \ref{LcpElliptic2}, let $P_0 \in \Jac(\cE)$ be the only element of order $2$ in $\Jac(\cE)$. Then for positive integers $0 < b-a < n$ such that $a$ and $b$ are odd numbers, we have
		$$(\C_\LL(D, a P_0 + (n-b) \cO), \, \C_\LL(D, b P_0 - a \cO ))$$
		is an LCP of elliptic codes.
	\end{corollary}
	\begin{proof}
		Let $P_0 \in \Jac(\cE)$ be an element of order $\alpha=2$. Since $a$ and $b$ are odd, $\gcd(\alpha, b+1) \neq 1 = \frac{\alpha}{\alpha}$ and $a \not\equiv 0 \pmod{\alpha}$, so the result follows by Theorem \ref{LcpElliptic2}.
	\end{proof}
	\begin{corollary}
		\label{Corell4}
		With the same assumptions of the Theorem \ref{LcpElliptic2}, let $P_0 \in \Jac(\cE)$ be an element of odd order greater than $1$ in $\Jac(\cE)$. Then for positive integers $1 < b < n+1$, we have
		$$(\C_\LL(D,  P_0 + (n-b) \cO), \, \C_\LL(D, b P_0 - \cO ))$$
		is an LCP of elliptic codes. In particular, we obtain LCPs with $\CL(D, G)$ having all possible dimensions $1 \le k \le n-1$.
	\end{corollary}
	\begin{proof}
		Take $a=1$ in Theorem \ref{LcpElliptic2} and observe that $\C_\LL(D,  P_0 + (n-b) \cO)$ is an $[n, n+1-b]$-code over $\F_q$.
	\end{proof}

	\begin{proposition}
		\label{propE1} Let $p > 3$ be a prime such that $p \equiv 3 \pmod{4}$. Let $\cE/\F_p$ be the elliptic function field of the curve $y^2=x^3+x$ over $\F_p$. Assume that $P_0 \in \Jac(\cE)$ is the common zero of $x$ and $y$ in $\cE$. For $a$, $b$ positive integers such that $0 < b - a < p-1$, $a$ and $b$ are odd, define
		\begin{align*}
			D & =\sum_{P \in \Jac(\cE) \setminus \{\cO, P_0\}} P, \\
			G & = aP_0 + (p-1-b)\cO \quad \mbox{and} \quad H  = bP_0 - a\cO.
		\end{align*}
		Then $\left( \CL(D,G), \CL(D, H) \right)$ is an LCP of codes over $\F_p$.
		
	\end{proposition}
	\begin{proof}
		For $p \equiv 3 \pmod{4}$, the elliptic curve given by $y^2=x^3+x$ over $\F_p$ is such that $\#\Jac(\cE)=p+1$ and $\Jac(\cE)$ contains only one element of order $2$, since $-1$ is not a quadratic residue modulo $p$, see \cite[Example 4.5]{silverman2009}. Therefore, the result follows from Theorem \ref{LcpElliptic2}.
	\end{proof}
	We have constructed codes $\CL(D,G)$ over $\F_q$ from elliptic curves with $\Jac(\cE)=\{ \cO, P, P_1,$ $\dots, P_n \}$ and $D=P_1 + \cdots + P_n$. According to \cite[Theorem 6.3]{JANWA1990}, for $q+1 < n < \#\Jac(\cE)$, the elliptic code $\CL(D,G)$ is an optimal code. Hence, there exist many different LCPs of optimal elliptic codes since elliptic curves have more than $q+2$ $\F_q$-rational points.

	\begin{example}
		Let $\cE/\F_8$ be the function field of the elliptic curve $y^2 + y = x^3 + x + 1$ over $\F_8$. We have $\Jac(\cE)$ is a group of order $13$, so it is a cyclic group and every $P \in \Jac(\cE) \setminus \{\cO\}$ has order $13$. Let $P_0 \in \Jac(\cE)\setminus \{\cO \}$ and $D=\sum_{P \in \Jac(\cE)\setminus\{\cO, P_0\} }$.
		
		Let $G=P_0+(11-b)\cO$ and $H=bP_0-\cO$.
		From Theorem \ref{LcpElliptic1} and \cite[Example 4.3]{munuera93}, for any $2 \le b \le 10$ there are LCPs of codes $(\CL(D,G), \CL(D,H))$ with parameters
		$[11,12-b, b-1]$ and  $[11, b-1, 12-b]$ over $\F_8$. For $b=11$, we can also verify that $\CL(D, P_0+9\cO)$ is a $[11, 1, 11]$-code and $\CL(D, 2P_0-\cO)$ is a $[11, 10, 2]$-code. Therefore, 
		$$
		\left( \CL(D, P_0+(11-b)\cO), \CL(D, bP_0-\cO) \right) 
		$$
		is an LCP of non-MDS optimal elliptic codes for any $2 \le b \le 10$ and an LCP of MDS elliptic optimal codes for $b=11$.
	\end{example}
	%\lu{arrumar a ref. da Gatti}
	%STYLE:
	%\bibliographystyle{alpha}
	%\bibliographystyle{amsalpha}
	\bibliographystyle{abbrv}
	%\bibliographystyle{acm}
	%\bibliographystyle{apalike}
	%\bibliographystyle{ieeetr}
	%\bibliographystyle{plain}
	%\bibliographystyle{siam}
	%\bibliographystyle{unsrt}

	%\bibliography{mybib}

\end{document}